\documentclass[10pt,b5paper]{article}

\usepackage[T1]{fontenc}
\usepackage[utf8]{inputenc}
\usepackage[english]{babel}
\usepackage{amsmath,amssymb,amsthm}
\usepackage{lmodern}
\usepackage{geometry}
\usepackage{url}
\usepackage{titlesec}
\titleformat{\paragraph}[runin]{\normalfont\itshape}{}{0pt}{}[.]
\titlespacing*{\paragraph}{0pt}{0.7em}{0.5em}

\newtheorem{theorem}{Theorem}[section]

\newtheorem{proposition}[theorem]{Proposition}
\newtheorem{corollary}[theorem]{Corollary}
\theoremstyle{definition}

\theoremstyle{remark}
\newtheorem{remark}[theorem]{Remark}
\numberwithin{equation}{section}

\DeclareMathOperator{\E}{E}
\DeclareMathOperator{\Var}{Var}
\DeclareMathOperator{\Bin}{Bin}
\newcommand{\R}{\mathbb{R}}

\title{Absolute deviations of the binomial\\ about a prescribed centre:\\
the tail expansion in closed form}
\author{
N.~Elezovi\'c\\
Department of Applied Mathematics,\\
Faculty of Electrical Engineering and Computing,\\
University of Zagreb, 10000 Zagreb, Croatia\\
\texttt{neven.elezovic@fer.hr}
}
\date{}

\begin{document}
\maketitle

\begin{abstract}
	Let $X\sim\Bin(N,p)$ and let $c=Nr$ be a prescribed centre with $r\ne p$.  The finite formula
	for $\E|X-c|$ is classical, and the existence of a complete large-deviation expansion in powers
	of $N^{-1}$ follows from Timashev's saddle-point expansion for binomial tails.  We give the
	coefficients in closed form for the stop-loss tail produced by the absolute deviation, with the
	lattice defect of $Nr$ retained explicitly.

	The construction combines three standard ingredients: Bernoulli polynomials for the logarithm of
	successive binomial masses, the exponential Bell recursion, and Eulerian polynomials for weighted
	geometric sums.  The $k$-th coefficient has denominator exactly $(1-\rho)^{2k+2}$, where $\rho$
	is the exponential tilt; hence the effective expansion parameter is $1/(N(1-\rho)^2)$, the
	lattice weighted analogue of the parameter in the incomplete gamma expansion.  The case
	$r=\tfrac12$ gives the fold at the origin of a biased $\pm1$ random walk, and its parity split is
	identified with the lattice defect of the fold.
\end{abstract}



% =====================================================================
\section{Introduction}
\label{sec:intro}

	Let $X\sim\Bin(N,p)$, $q=1-p$, and consider the folded binomial $|X-c|$. Three
	centres occur naturally, and they are three different problems.

	\begin{itemize}
		\item \textbf{The mean, $c=Np$.} This is the mean absolute deviation, one of
		the oldest results in probability.
		\item \textbf{A prescribed centre, $c=Nr$ with $r\ne p$.} The fold now sits
		in the tail of the law, at a distance of order $N$ from the bulk.
		\item \textbf{The origin, $c=N/2$.} With $S_N=2X-N$ the position of the
		$\pm1$ walk that steps $+1$ with probability $p$, one has $|S_N|=2|X-N/2|$,
		so this is the expected absolute displacement of a biased walk. It is the
		second case with $r=\tfrac12$: the origin is the hypothesis that the coin is
		fair.
	\end{itemize}

	\paragraph{What is classical}
	The finite theory and the qualitative asymptotics are classical.

	The mean-centred case was put to De~Moivre by Alexander Cuming in 1721 and
	answered, for a fair coin, in the \emph{Miscellanea Analytica} of 1730; the
	biased version is Problem~73 of the third edition of the \emph{Doctrine of
	Chances}. With $\nu$ the integer satisfying $Np<\nu\le Np+1$,
	\begin{equation}\label{eq:demoivre-mean}
		\E|X-Np|=2\nu q\,b(\nu;N,p),\qquad b(k;N,p)=\binom Nk p^kq^{N-k}.
	\end{equation}
	Proofs were supplied by Todhunter~\cite{todhunter1865} and, without the
	assumption that $Np$ be an integer, by Poincar\'e~\cite{poincare1896}; it is
	Problem~35 of Chapter~IX of Feller~\cite[p.~241]{feller1968}, and the history is
	told by Diaconis and Zabell~\cite{diaconis_zabell}.

	For an \emph{arbitrary} centre the answer is equally classical. Todhunter's
	identity, in the form recorded as Lemma~1 of \cite[p.~288]{diaconis_zabell},
	gives $\E|X-c|$ in one distribution function and one point mass
	(Theorem~\ref{thm:todhunter} below), and Katti~\cite{katti1960} obtained the
	absolute moments of the binomial about an arbitrary constant by an incomplete
	generating function; the result is recorded in Johnson, Kotz and
	Kemp~\cite[\S3.3]{jkk}. The variable $|2X-N|$ is itself a named distribution,
	the \emph{folded binomial}~\cite{gart1970}. We do not claim novelty for the
	finite identities used below.

	The asymptotics for fixed $r\ne p$ is a large-deviation problem, and it too has
	been treated. The leading behaviour is Bahadur--Rao~\cite{bahadur_rao}; the
	non-lattice-point correction $\rho^{\{Nr\}}$ that appears below is the lattice
	adjustment of the saddle-point approximation to $\E[(X-K)^+]$ given by Huang and
	Oosterlee~\cite{huang_oosterlee}; and the \emph{existence} of a complete
	expansion in powers of $N^{-1}$ is due to Timashev~\cite{timashev1999}, who
	obtained it for binomial and Poisson tail probabilities by the saddle-point
	method, with coefficients defined by a recurrence rather than in closed form.
	Finally, the transition to the mean-centred regime, as $r\to p$, is Cram\'er's
	moderate-deviation theorem (Petrov~\cite[Ch.~VIII]{petrov1975}), and a
	description that is \emph{uniform} across the whole transition is available
	through Temme's uniform expansion of the incomplete beta
	function~\cite[\S8.18]{dlmf},~\cite{temme2015}.

	\paragraph{What is done here}
	The remaining point is the explicit form of the coefficients.

	\begin{itemize}
		\item \textbf{The coefficients in closed form}
		(Theorem~\ref{thm:all-coeff}). The expansion is generated by a three-step
		recursion, each step classical in isolation: Bernoulli polynomials expand the
		logarithm of the ratio of successive binomial masses; an exponential (Bell)
		recursion exponentiates it; Eulerian polynomials sum the resulting weighted
		geometric series. The weight is the linear factor produced by the absolute
		value --- so this is a \emph{stop-loss} tail, not a tail probability --- and
		the lattice defect $\theta_N=\{Nr\}$ of the centre is carried explicitly.
		We have not found the coefficients written out anywhere.
		\item \textbf{The exact denominators} (Corollary~\ref{cor:parameter}). The
		$k$-th coefficient has denominator exactly $(1-\rho)^{2k+2}$, with
		$(1-\rho)^{2k+2}S_k\to(-1)^k(2k+1)!!/(r\bar r)^k$. Hence the expansion
		proceeds in $1/(N(1-\rho)^2)$ and not in $1/N$. This is the lattice, weighted
		counterpart of a familiar fact: in the expansion of the incomplete gamma
		function at $z=\lambda a$ the successive terms carry $(z-a)^{-(2k+1)}$, i.e.\
		the parameter $1/(a(1-\lambda)^2)$ \cite[\S8.11(iii)]{dlmf}. The one extra
		power is what the linear weight contributes.
		\item \textbf{The walk} (Corollary~\ref{cor:walk-asymp},
		Remark~\ref{rem:walk-reading}). Taking $r=\tfrac12$ gives the asymptotics of
		the absolute displacement of a biased walk, and shows that the even/odd
		dichotomy of the exact formula is the lattice effect: $\theta_v=\{v/2\}$ takes
		exactly the two values $0$ and $\tfrac12$, so the parity of $v$ records the
		lattice defect of the fold.
		\item \textbf{The two orientations of the hypergeometric form}
		(Remark~\ref{rem:two-routes}). One might hope to expand the exact walk
		formula by multiplying the Stirling expansion of a binomial coefficient by a
		terminating ${}_2F_1$. This works only when the argument of the ${}_2F_1$ is a
		contraction; in that orientation it is equivalent to the tilt expansion. In
		the opposite orientation the hypergeometric function is exponentially large,
		the binomial coefficient exponentially small, their rates cancel exactly, and
		the desired quantity is not captured by termwise multiplication of the two
		asymptotic expansions.
	\end{itemize}

	Throughout, $b$ and $B$ denote the binomial mass and distribution function,
	\[
		b(y;k,\vartheta)=\binom ky \vartheta^y(1-\vartheta)^{k-y},\qquad
		B(\ell;k,\vartheta)=\Pr\{\Bin(k,\vartheta)\le\ell\},
	\]
	with $B(\ell;k,\vartheta)=0$ for $\ell<0$ and $=1$ for $\ell\ge k$.

% =====================================================================
\section{The classical finite formulas}
\label{sec:engine}

	The truncated first moment of a binomial is a binomial distribution function,
	\begin{equation}\label{eq:partial-first-moment}
		\sum_{y\le m}y\,b(y;k,\vartheta)=k\vartheta\,B(m-1;k-1,\vartheta),
	\end{equation}
	immediate from $y\binom ky=k\binom{k-1}{y-1}$. Combined with $|x|=x+2x^-$ this
	gives the absolute deviation about an arbitrary centre. In actuarial language
	the object is the \emph{stop-loss transform} $\E[(c-Y)^+]$; the underlying
	identity is Todhunter's.

\begin{theorem}[Todhunter's form; \cite{todhunter1865},
	{\cite[Lemma~1]{diaconis_zabell}}, \cite{katti1960}]\label{thm:todhunter}
	Let $Y\sim\Bin(k,\vartheta)$ with $0<\vartheta<1$, let $c\in\R$, and put
	$m:=\lfloor c\rfloor$. Then
	\begin{equation}\label{eq:todhunter}
		\E|Y-c| = (k\vartheta-c)\bigl[1-2B(m;k,\vartheta)\bigr]
		+ 2(m+1)(1-\vartheta)\,b(m+1;k,\vartheta).
	\end{equation}
\end{theorem}

\begin{proof}
	Writing $|Y-c|=(Y-c)+2(c-Y)^+$ and taking expectations,
	$\E|Y-c|=(k\vartheta-c)+2\,\E[(c-Y)^+]$. Only $y<c$ contributes to the partial
	expectation, and if $c$ is an integer the endpoint $y=c$ contributes zero, so
	the sum may be written through $m=\lfloor c\rfloor$ in all cases; the
	conventions on $B$ cover $m<0$ and $m\ge k$. Split the summand as
	$c-y=(k\vartheta-y)+(c-k\vartheta)$. The second part contributes
	$(c-k\vartheta)B(m;k,\vartheta)$. For the first, use
	\begin{equation}\label{eq:sbp}
		(k\vartheta-y)\,b(y;k,\vartheta)
		=k\vartheta(1-\vartheta)\bigl[b(y;k-1,\vartheta)-b(y-1;k-1,\vartheta)\bigr],
	\end{equation}
	which follows from $y\,b(y;k,\vartheta)=k\vartheta\,b(y-1;k-1,\vartheta)$
	together with the one-step recursion
	$b(y;k,\vartheta)=\vartheta\,b(y-1;k-1,\vartheta)
	+(1-\vartheta)\,b(y;k-1,\vartheta)$. (The factor $1-\vartheta$ is essential and
	is easily dropped by mistake.) Summing \eqref{eq:sbp} over $y\le m$ telescopes
	to $k\vartheta(1-\vartheta)\,b(m;k-1,\vartheta)$, which equals
	$(m+1)(1-\vartheta)\,b(m+1;k,\vartheta)$ by \eqref{eq:partial-first-moment} read
	at a single point. Collecting the two occurrences of $(k\vartheta-c)$ gives
	\eqref{eq:todhunter}.
\end{proof}

\begin{corollary}[The mean-centred one-term form]
	\label{cor:collapse}
	In \eqref{eq:todhunter} the distribution function disappears precisely when
	$c=k\vartheta$: the term carrying it is
	\[
		(k\vartheta-c)\bigl[1-2B(m;k,\vartheta)\bigr],
	\]
	and $B(m;k,\vartheta)=\tfrac12$ does not hold in general. At $c=k\vartheta$, with $\nu$ the integer satisfying
	$k\vartheta<\nu\le k\vartheta+1$,
	\begin{equation}\label{eq:collapse}
		\E|Y-k\vartheta| = 2\nu(1-\vartheta)\,b(\nu;k,\vartheta)
		= 2k\vartheta(1-\vartheta)\,b(\nu-1;k-1,\vartheta),
	\end{equation}
	which is \eqref{eq:demoivre-mean}.
\end{corollary}

	This is a statement about the decomposition \eqref{eq:todhunter}. It does not
	assert that no single-term representation of some other kind can exist at
	another centre, and we make no such claim; Diaconis and Zabell remark
	\cite[p.~284]{diaconis_zabell} that the truncated binomial sum admits ``no
	essential simplifications'', citing Zeilberger~\cite{zeilberger1990}.

	\paragraph{The mean-centred asymptotics} These are not needed for what follows,
	but they are needed for the applications, and we record them for reference.
	Write $h_N:=\lceil Np\rceil-Np\in[0,1)$ for the lattice defect of the mean,
	$s^2:=4pq$, $\nu:=\lceil Np\rceil$ and $B_2(t)=t^2-t+\tfrac16$. The companion
	note~\cite[Theorem~4.2]{elezovic_mad} gives a complete expansion in integer
	powers of $N^{-1}$ whose coefficients oscillate through $h_N$; its first orders
	are
	\begin{equation}\label{eq:mean-asymp}
		2\,\E|X-Np|=\sqrt{\frac{2N}{\pi}}\,s
		\Bigl[\,1+\frac{\tilde c_1(N)}N+O(N^{-2})\Bigr],
		\qquad
		\tilde c_1(N)=\frac1{12}-\frac{B_2(h_N)}{2pq}.
	\end{equation}
	Since $\E(2|X-Np|)^2=4Npq$ exactly, the variance follows at once from
	\eqref{eq:collapse}:
	\begin{equation}\label{eq:mean-var}
	\begin{aligned}
		\Var\bigl(2|X-Np|\bigr)
		&=4Npq-16\,\nu^2q^2\,b(\nu;N,p)^2\\
		&=4pq\Bigl(1-\frac2\pi\Bigr)N-\frac{4pq}{3\pi}
		+\frac8\pi B_2(h_N)+O(N^{-1}).
	\end{aligned}
	\end{equation}
	At $p=\tfrac12$ the oscillation becomes the parity of $N$, and the constant
	order of \eqref{eq:mean-var} is $(-1)^N/\pi$.

% =====================================================================
\section{The fold in the tail: the expansion in closed form}
\label{sec:hyp}

	Let the centre now be a prescribed $Nr$ with $r\ne p$, and put
	\[
		A:=2\,|X-Nr| .
	\]
	Two things change. First, by Corollary~\ref{cor:collapse} there is no collapse.
	Second, the fold has left the bulk of the law, so the correction to the
	deterministic value $2N|p-r|$ is exponentially small, and its size is governed
	by a large-deviation rate. Write $\bar r:=1-r$ and
	\[
		\theta_N:=\{Nr\}\in[0,1),\qquad
		\rho:=\frac{r\,q}{\bar r\,p},\qquad
		D(r\Vert p):=r\log\frac rp+\bar r\log\frac{\bar r}q>0 .
	\]
	Note that $\rho<1$ exactly when $r<p$. The relabelling $X\mapsto N-X$,
	$p\mapsto q$, $r\mapsto\bar r$ leaves $A$ invariant, so \emph{every asymptotic
	statement below is written after that relabelling, and $r<p$, $\rho<1$
	throughout.}

% ---------------------------------------------------------------------
\subsection{Exact reduction}

\begin{proposition}[Exact form, and the first coefficients]\label{prop:hyp-exp}
	Let $0<p<1$, $0<r<1$, $X\sim\Bin(N,p)$ and $m:=\lfloor Nr\rfloor$.

	\emph{(i)} For \emph{every} $r$, Theorem~\ref{thm:todhunter} at $c=Nr$, $k=N$,
	$\vartheta=p$ gives
	\begin{equation}\label{eq:hyp-exact}
		\E A=2N(p-r)\bigl[1-2B(m;N,p)\bigr]+4(m+1)\,q\,b(m+1;N,p),
	\end{equation}
	equivalently $\E A=2N(p-r)+4\,\E[(Nr-X)^+]$.

	\emph{(ii)} Write the tail through its largest term:
	\begin{equation}\label{eq:T-def}
		\E\bigl[(Nr-X)^+\bigr]=b(m;N,p)\,T_N,
		\quad
		T_N:=\sum_{j=0}^{m}(\theta_N+j)\,R_j,
		\quad
		R_j:=\frac{b(m-j;N,p)}{b(m;N,p)} .
	\end{equation}
	Then, for fixed $r<p$,
	\begin{equation}\label{eq:hyp-tail}
		T_N=S_0(\theta_N)-\frac{U_1(\theta_N)}{N\,r\bar r}+O(N^{-2}),
	\end{equation}
	where, with $M_i:=\sum_{j\ge0}j^{i}\rho^{\,j}$ and $c:=r-\tfrac12$,
	\begin{align}
		S_0(\theta)&=\theta M_0+M_1,\label{eq:S0}\\
		U_1(\theta)&=\theta(\theta+c)\,M_1
		 +\Bigl(\tfrac32\theta+c\Bigr)M_2+\tfrac12M_3,\label{eq:U1}
	\end{align}
	and $M_0=(1-\rho)^{-1}$, $M_1=\rho(1-\rho)^{-2}$,
	$M_2=\rho(1+\rho)(1-\rho)^{-3}$, $M_3=\rho(1+4\rho+\rho^2)(1-\rho)^{-4}$.

	\emph{(iii)} The local mass factors \emph{exactly} as
	\begin{equation}\label{eq:hyp-tilt}
		b(m;N,p)=e^{-N D(r\Vert p)}\,\rho^{\,\theta_N}\,\pi_N(-\theta_N;\,r),
		\qquad
		\pi_N(h;r):=\binom{N}{Nr+h}r^{\,Nr+h}\,\bar r^{\,N\bar r-h},
	\end{equation}
	so that, by the local expansion \cite[Theorem~3.1]{elezovic_mad},
	\begin{equation}\label{eq:hyp-asymp}
		\E A=2N(p-r)
		+\frac{4\,e^{-N D(r\Vert p)}\,\rho^{\,\theta_N}}{\sqrt{2\pi N\,r\bar r}}
		\,S_0(\theta_N)\,\bigl(1+O(N^{-1})\bigr).
	\end{equation}
	The correction to the deterministic term is therefore of order
	$N^{-1/2}e^{-ND(r\Vert p)}$.

	The remainders in \emph{(ii)} and \emph{(iii)} are uniform in
	$\theta_N\in[0,1)$ and uniform for $(p,r)$ in compact subsets of
	$\{0<r<p<1\}$; they are \emph{not} uniform as $r\to p$
	(Corollary~\ref{cor:parameter}).
\end{proposition}

	Part~(iii) is Bahadur--Rao~\cite{bahadur_rao} with the lattice adjustment
	$\rho^{\theta_N}$, which for the stop-loss functional is the correction of Huang
	and Oosterlee~\cite{huang_oosterlee}. The exponential tilt is visible in
	\eqref{eq:hyp-tilt}: tilting $\Bin(N,p)$ so as to move its mean to $Nr$ requires
	$e^{t_*}=\rho$, the rate $D(r\Vert p)$ is the Legendre transform, and
	$\rho^{\theta_N}$ is the price of $Nr$ not being a lattice point. It is now the
	\emph{centre}, and not the data law, that sets the rhythm of the oscillation.

\begin{proof}
	(i) is Theorem~\ref{thm:todhunter}, doubled. For (ii), the ratio telescopes,
	\[
		R_j=\prod_{i=1}^{j}\frac{(m-i+1)q}{(N-m+i)p},
	\]
	and since $m=Nr-\theta_N$ and $N-m=N\bar r+\theta_N$, each factor equals
	\begin{equation}\label{eq:tail-factor}
		\rho\cdot\frac{1-(\theta_N+i-1)/(Nr)}{1+(\theta_N+i)/(N\bar r)} .
	\end{equation}
	The numerator bracket is at most $1$ and the denominator bracket at least $1$,
	so every factor is at most $\rho$ and
	\begin{equation}\label{eq:tail-geom}
		R_j\le\rho^{\,j}\qquad\text{for all }0\le j\le m ,
	\end{equation}
	an unconditional bound that will do all the work of controlling the far tail.
	The rest of (ii) is the case $K=1$ of Theorem~\ref{thm:all-coeff}.

	(iii) The factorisation is exact algebra:
	\[
		b(m;N,p)=\bigl[\binom Nm r^m\bar r^{\,N-m}\bigr]
		\Bigl(\frac pr\Bigr)^{m}\Bigl(\frac q{\bar r}\Bigr)^{N-m},
	\]
	the bracket being $\pi_N(-\theta_N;r)$ because $m=Nr-\theta_N$; splitting the
	exponents at $m=Nr-\theta_N$ gives
	$(p/r)^{Nr}(q/\bar r)^{N\bar r}=e^{-ND(r\Vert p)}$ and
	$(p/r)^{-\theta_N}(q/\bar r)^{\theta_N}=\rho^{\theta_N}$. The local expansion
	\cite[Theorem~3.1]{elezovic_mad} gives
	$\pi_N(-\theta_N;r)=(2\pi Nr\bar r)^{-1/2}(1+O(N^{-1}))$, uniformly in
	$\theta_N$ and uniformly for $r$ in compact subsets of $(0,1)$.
\end{proof}

% ---------------------------------------------------------------------
\subsection{All the coefficients}
\label{ssec:all-coefficients}

	That binomial tail \emph{probabilities} in this regime admit a complete expansion in
	powers of $N^{-1}$, with coefficients given by a recurrence, is known: it is due to
	Timashev~\cite{timashev1999} by the saddle-point method, and is a special case of the
	general large-deviation expansions of Fill~\cite{fill1989} and, in the lattice setting,
	of the mod-$\phi$ framework of F\'eray, M\'eliot and Nikeghbali~\cite{fmn2016}, whose
	expansion already carries the prefactor $1/(1-\rho)$. Nor is the \emph{weighted} tail
	\eqref{eq:T-def} independent of the unweighted one: the reduction \eqref{eq:hyp-exact}
	writes it through binomial tail probabilities, so any complete expansion of those yields
	one for the weighted tail --- though there the two contributions exceed their difference
	by a factor $\asymp N(1-\rho)$, and the reduction therefore consumes one asymptotic order.
	What the next theorem adds is not the existence of the expansion but the \emph{closed
	form} of its coefficients, obtained directly for the weighted (stop-loss) tail; a
	term-by-term comparison with Timashev's coefficients is not immediate, his being defined
	by a recurrence in the complex domain. Let $B_k$
	denote the Bernoulli polynomials, let
	\begin{equation}\label{eq:Xi-def}
		\Xi_k(r):=r^{-k}+(-1)^{k+1}\bar r^{-k},
	\end{equation}
	and recall that $M_i=\sum_{j\ge0}j^{i}\rho^{\,j}=\rho\,A_i(\rho)/(1-\rho)^{i+1}$
	for $i\ge1$, with $A_i$ the Eulerian polynomial, and $M_0=(1-\rho)^{-1}$.

\begin{theorem}[The coefficients in closed form]\label{thm:all-coeff}
	Fix $0<r<p<1$, and define polynomials $P_k,Q_k$ in $j$, with coefficients
	depending on $\theta$ and $r$, by
	\begin{equation}\label{eq:Pk}
		P_k(j):=-\frac{\Xi_k(r)}{k(k+1)}\Bigl[B_{k+1}(\theta+j)-B_{k+1}(\theta)\Bigr]
		-\frac{(-1)^{k+1}}{k\,\bar r^{\,k}}\Bigl[(\theta+j)^k-\theta^k\Bigr],
		\qquad k\ge1,
	\end{equation}
	and
	\begin{equation}\label{eq:Qk}
		Q_0:=1,\qquad
		k\,Q_k=\sum_{i=1}^{k}i\,P_i\,Q_{k-i}\qquad(k\ge1).
	\end{equation}
	Then $\deg_j P_k\le k+1$, with equality unless $\Xi_k(r)=0$ --- that is, unless
	$r=\tfrac12$ and $k$ is even --- while $\deg_j Q_k=2k$ in all cases. Put
	\begin{equation}\label{eq:Sk}
		S_k(\theta):=\sum_{j\ge0}(\theta+j)\,Q_k(j)\,\rho^{\,j},
	\end{equation}
	a finite linear combination of $M_0,\dots,M_{2k+1}$; it is a polynomial in
	$\theta$ whose coefficients are rational in $\rho$ and in $r$. Then, for every
	fixed $K\ge0$, as $N\to\infty$,
	\begin{equation}\label{eq:T-expansion}
		T_N=\sum_{k=0}^{K}\frac{S_k(\theta_N)}{N^{k}}+O\bigl(N^{-(K+1)}\bigr),
	\end{equation}
	uniformly in $\theta_N\in[0,1)$ and uniformly for $(p,r)$ in compact subsets of
	$\{0<r<p<1\}$; the constant depends on $K$ and on that compact set, and
	deteriorates as $\rho\uparrow1$ (Corollary~\ref{cor:parameter}). The first two
	coefficients are
	\[
		S_0(\theta)=\theta M_0+M_1,
		\qquad
		S_1(\theta)=-\frac{U_1(\theta)}{r\bar r},
	\]
	with $U_1$ as in \eqref{eq:U1}; note that $U_1$ is the raw moment combination
	and $S_1$ the coefficient itself.
\end{theorem}

\begin{proof}
	\emph{Step 1: an exact logarithm.} By \eqref{eq:tail-factor},
	\[
		\log R_j=j\log\rho
		-\sum_{k\ge1}\frac1k
		\Bigl[(Nr)^{-k}\sum_{i=1}^{j}(\theta+i-1)^k
		+(-1)^{k+1}(N\bar r)^{-k}\sum_{i=1}^{j}(\theta+i)^k\Bigr].
	\]
	The inner sums are power sums, and Bernoulli's formula
	$\sum_{u=0}^{n-1}(x+u)^k=[B_{k+1}(x+n)-B_{k+1}(x)]/(k+1)$ gives
	\begin{gather*}
		\sum_{i=1}^{j}(\theta+i-1)^k
		=\frac{B_{k+1}(\theta+j)-B_{k+1}(\theta)}{k+1},\\
		\sum_{i=1}^{j}(\theta+i)^k
		=\frac{B_{k+1}(\theta+j+1)-B_{k+1}(\theta+1)}{k+1}.
	\end{gather*}
	Substituting these, and eliminating the shifted Bernoulli polynomial by
	$B_{k+1}(x+1)=B_{k+1}(x)+(k+1)x^k$, collects the two contributions into the one
	combination $\Xi_k(r)$ and produces exactly \eqref{eq:Pk}:
	\begin{equation}\label{eq:logR-P}
		\log R_j=j\log\rho+\sum_{k\ge1}N^{-k}P_k(j).
	\end{equation}
	This is an identity; no approximation has been made.

	\emph{Step 2: the range that matters.} Put $J:=\lceil(\log N)^2\rceil$. By
	\eqref{eq:tail-geom},
	\[
		\sum_{j>J}(\theta_N+j)R_j\le\sum_{j>J}(1+j)\rho^{\,j}
		=O\bigl(J\rho^{J}\bigr)=O\bigl(e^{-c_0(\log N)^2}\bigr)
	\]
	for some $c_0>0$, which is smaller than every power of $N^{-1}$; the same bound,
	with $Q_k$ in place of $1$, removes the tail of \eqref{eq:Sk}. It therefore
	suffices to work with $j\le J$ and to restore the sums to $j=\infty$ at the end.
	The cut at $J$ is necessary: at $j\asymp\sqrt N$ one has
	$N^{-1}P_1(j)=\Theta(1)$, and the exponential could not be expanded term by term
	there.

	\emph{Step 3: exponentiation.} For $j\le J$ the terms of \eqref{eq:logR-P}
	beyond $k=K+1$ are $O\bigl((\log N)^{2K+6}N^{-(K+2)}\bigr)$, uniformly in
	$\theta$, and $N^{-k}P_k(j)=O\bigl((\log N)^{2k+2}N^{-k}\bigr)=o(1)$; so the
	exponential may be expanded. The exponential formula
	\[
		\exp\Bigl\{\sum_{k\ge1}x^kP_k\Bigr\}=\sum_{k\ge0}x^kQ_k
	\]
	--- equivalently the recursion \eqref{eq:Qk}, obtained by differentiating that identity in $x$ --- gives
	\[
		R_j=\rho^{\,j}\Bigl[\sum_{k=0}^{K}N^{-k}Q_k(j)
		+O\bigl((\log N)^{4K+8}N^{-(K+1)}\bigr)\Bigr],
		\qquad j\le J,
	\]
	the power of $\log N$ being governed by $\deg_j Q_{K+1}=2K+2$. Multiplying by
	$(\theta_N+j)$, summing over $j\le J$ and restoring the tails by Step~2 gives
	\eqref{eq:T-expansion} with remainder $O((\log N)^{4K+8}N^{-(K+1)})$. Applying
	this with $K+1$ in place of $K$, and absorbing the last displayed term, gives
	the clean $O(N^{-(K+1)})$: the logarithmic factors are harmless precisely
	because $K$ is arbitrary.

	\emph{Step 4: the degrees.} That $\deg_jP_k\le k+1$ is clear from
	\eqref{eq:Pk}, with equality unless the coefficient $\Xi_k(r)$ of the leading
	Bernoulli term vanishes; and $\Xi_k(r)=0$ forces $r^{-k}=(-1)^{k}\bar r^{-k}$,
	hence $k$ even and $r=\bar r=\tfrac12$. Since $\Xi_1(r)=1/r+1/\bar r\ne0$
	always, $\deg_jP_1=2$; the dominant contribution to \eqref{eq:Qk} is
	$P_1^{\,k}/k!$, so $\deg_jQ_k=2k$ in all cases. Hence $(\theta+j)Q_k(j)$ has
	degree $2k+1$ and \eqref{eq:Sk} is a combination of $M_0,\dots,M_{2k+1}$.
	Finally $Q_1=P_1=-\frac1{r\bar r}\bigl[\tfrac12j^2+(\theta+c)j\bigr]$, whose
	pairing with $(\theta+j)\rho^j$ is $-U_1/(r\bar r)$.
\end{proof}

\begin{remark}[The three ingredients]\label{rem:ingredients}
	Each step is classical in isolation, and that is the point. The Bernoulli
	polynomials arise as the power sums of Step~1, which is why the combination
	$\Xi_k(r)$ of \eqref{eq:Xi-def} appears --- the same combination that governs
	the local expansion of the binomial mass~\cite{elezovic_mad}. The recursion
	\eqref{eq:Qk} is the exponential (Bell) formula. The moments $M_i$ are Eulerian
	polynomials over $(1-\rho)^{i+1}$. The contribution is the combined closed-form
	recursion for the coefficients.

	At $r=\tfrac12$ --- the walk of \S\ref{sec:origin} --- one has
	$\Xi_k(\tfrac12)=2^k\bigl(1+(-1)^{k+1}\bigr)$, so that
	\begin{equation}\label{eq:Xi-half}
		\Xi_{2i}(\tfrac12)=0,\qquad \Xi_{2i+1}(\tfrac12)=2^{2i+2}:
	\end{equation}
	half of the Bernoulli contributions vanish identically. Thus the additional
	symmetry of a half-integer fold is reflected directly in the coefficients.
\end{remark}

% ---------------------------------------------------------------------
\subsection{The denominators, and the expansion parameter}

\begin{corollary}[Exact denominators]\label{cor:parameter}
	For each $k\ge0$ the coefficient $S_k(\theta)$ has denominator exactly
	$(1-\rho)^{2k+2}$; more precisely, uniformly in $\theta\in[0,1)$,
	\begin{equation}\label{eq:Sk-limit}
		\lim_{\rho\uparrow1}(1-\rho)^{2k+2}\,S_k(\theta)
		=\frac{(-1)^k\,(2k+1)!!}{(r\bar r)^{k}}\ \ne\ 0 .
	\end{equation}
	Consequently the ratio of the $k$-th term of \eqref{eq:T-expansion} to the
	leading one is $\Theta(\lambda^{k})$, where
	\begin{equation}\label{eq:lambda}
		\lambda:=\frac{1}{N(1-\rho)^{2}} ,
	\end{equation}
	and \eqref{eq:T-expansion} is informative precisely while $\lambda\to0$.
\end{corollary}

\begin{proof}
	As $\rho\uparrow1$ one has $M_i\sim i!\,(1-\rho)^{-(i+1)}$, so in \eqref{eq:Sk}
	the highest power of $j$ dominates. By Step~4 of the proof of
	Theorem~\ref{thm:all-coeff} the leading coefficient of $Q_k$ is that of
	$P_1^{\,k}/k!$, namely $(-1/(2r\bar r))^{k}/k!$, so $(\theta+j)Q_k(j)$ has
	leading term $j^{2k+1}(-1/(2r\bar r))^{k}/k!$ and
	\[
		S_k(\theta)\sim\frac{1}{k!}\Bigl(\frac{-1}{2r\bar r}\Bigr)^{k}M_{2k+1}
		\sim\frac{(2k+1)!}{k!\,(-2r\bar r)^{k}}\cdot\frac1{(1-\rho)^{2k+2}} ,
	\]
	and $(2k+1)!/(k!\,2^k)=(2k+1)!!$. The limit is independent of $\theta$ and
	nonzero, which is the assertion. Dividing by $S_0\sim(1-\rho)^{-2}$ gives the
	ratio $\Theta((1-\rho)^{-2k})$, and multiplying by $N^{-k}$ gives
	$\Theta(\lambda^k)$.
\end{proof}

\begin{remark}[The parameter is the familiar one]\label{rem:parameter}
	Corollary~\ref{cor:parameter} is the lattice, weighted counterpart of a
	classical fact. In the asymptotic expansion of the
	incomplete gamma function at $z=\lambda a$ with $0<\lambda<1$,
	\[
		\gamma(a,z)\sim-z^ae^{-z}\sum_{k\ge0}
		\frac{(-a)^k\,b_k(\lambda)}{(z-a)^{2k+1}}
		\qquad\cite[\text{Eq.\ 8.11.6}]{dlmf},
	\]
	one has $z-a=-a(1-\lambda)$, so that the $k$-th term carries
	$\{a(1-\lambda)^2\}^{-k}$ relative to the leading one: the same parameter, with
	$a\leftrightarrow N$ and $\lambda\leftrightarrow\rho$. The exponent there is
	$2k+1$ and here it is $2k+2$; the one extra power is exactly what the linear
	weight of the stop-loss functional contributes. The double factorials have the corresponding shift:
	well: DLMF's coefficients satisfy $b_k(1)=(2k-1)!!$, against $(2k+1)!!$ in
	\eqref{eq:Sk-limit}. This comparison is included only to identify the expansion
	parameter; the new point here is the closed form of the coefficients themselves.
\end{remark}

\begin{remark}[The series diverges]\label{rem:divergent}
	Corollary~\ref{cor:parameter} concerns $\rho\uparrow1$ at fixed $k$. At fixed
	$\rho<1$ and $k\to\infty$ the coefficients behave differently, and it is worth
	saying that they grow. Since $M_i\sim i!\,(\log(1/\rho))^{-(i+1)}$ as
	$i\to\infty$, the same computation gives
	\[
		S_k\ \approx\ \frac{(2k+1)!}{k!\,(-2r\bar r)^{k}\,(\log(1/\rho))^{2k+2}} ,
	\]
	which grows like $4^k\,k!$ times a constant to the $k$. Thus
	\eqref{eq:T-expansion} is a \emph{divergent} asymptotic series, as such
	expansions usually are. Nothing here needs its convergence:
	\eqref{eq:T-expansion} is used only as a truncated expansion with a proved
	remainder, not as a convergent series.
\end{remark}

	The variance is immediate, since the second moment of a folded variable is the
	second moment of the variable itself.

\begin{proposition}[Variance about a prescribed centre]\label{prop:hyp-var}
	Exactly, for every $r$,
	$\Var(A)=4\bigl[Npq+N^2(p-r)^2\bigr]-(\E A)^2$; and for fixed $r<p$,
	\begin{equation}\label{eq:hyp-var-asymp}
		\Var(A)-4Npq
		=-16\,N(p-r)\,b(m;N,p)\,S_0(\theta_N)\,\bigl(1+O(N^{-1})\bigr).
	\end{equation}
	The correction to $4Npq$ is therefore negative; its modulus is of order
	$\sqrt N\,e^{-ND(r\Vert p)}$ and oscillates boundedly through $\theta_N$.
\end{proposition}

\begin{proof}
	$\E(A^2)=4\,\E(X-Nr)^2=4[\Var X+N^2(p-r)^2]$ exactly. Writing
	$\E|X-Nr|=N(p-r)+\varepsilon$ with $\varepsilon=2\,\E[(Nr-X)^+]>0$,
	\[
		\Var|X-Nr|=Npq-2N(p-r)\,\varepsilon-\varepsilon^2 ,
	\]
	and $\varepsilon=2\,b(m;N,p)[S_0(\theta_N)+O(N^{-1})]$ by
	Proposition~\ref{prop:hyp-exp}, while $\varepsilon^2$ is exponentially smaller
	again. That the correction is negative needs no computation: $|\cdot|$ is
	$1$-Lipschitz, so $\Var|W|\le\Var W$ for every $W$, and folding can only remove
	variance.
\end{proof}

% =====================================================================
\section{The origin is a hypothesis: the biased walk}
\label{sec:origin}

	Let $X\sim\Bin(v,p)$ count the $+1$ steps of a $\pm1$ walk, let $S_v=2X-v$ be
	its position, and put
	\begin{equation}\label{eq:F-def}
		F(v,p):=\E|S_v|=2\,\E\bigl|X-\tfrac v2\bigr| .
	\end{equation}
	Replacing $X$ by $v-X$ shows $F(v,p)=F(v,q)$, so it suffices to assume
	$p>\tfrac12$. Theorem~\ref{thm:todhunter} at $c=v/2$ gives the exact value at
	once: with $m_0:=\lfloor v/2\rfloor$,
	\begin{equation}\label{eq:F-cdf}
		F(v,p)=v(p-q)\bigl[1-2B(m_0;v,p)\bigr]+4(m_0+1)\,q\,b(m_0+1;v,p) .
	\end{equation}
	One distribution function and one point mass, for every $v$ and $p$. The formula
	is Todhunter's, and no novelty is asserted for the identity itself. The relevant
	point here is its asymptotic regime. The centre $v/2$ is the mean only when $p=\tfrac12$; for a
	biased walk it is a \emph{prescribed} centre, namely $r=\tfrac12$ --- the
	hypothesis that the coin is fair --- and \S\ref{sec:hyp} therefore applies to it
	verbatim.

\begin{corollary}[Asymptotics of the biased walk]\label{cor:walk-asymp}
	Let $p\ne\tfrac12$ be fixed; by $F(v,p)=F(v,q)$ we may assume $p>\tfrac12$. Put
	\[
		\rho:=\frac qp\in(0,1),\qquad
		\theta_v:=\Bigl\{\frac v2\Bigr\}=
		\begin{cases}0,&v\ \text{even},\\[2pt] \tfrac12,&v\ \text{odd}.\end{cases}
	\]
	Then, as $v\to\infty$,
	\begin{equation}\label{eq:walk-asymp}
		F(v,p)=v(p-q)
		+4\,\bigl(2\sqrt{pq}\,\bigr)^{v}\sqrt{\frac{2}{\pi v}}\;\rho^{\,\theta_v}
		\Bigl[\,\frac{\theta_v}{1-\rho}+\frac{\rho}{(1-\rho)^2}\,\Bigr]
		\bigl(1+O(v^{-1})\bigr),
	\end{equation}
	and the correction continues in integer powers of $v^{-1}$, its coefficients
	given by Theorem~\ref{thm:all-coeff} at $r=\tfrac12$. The exponential rate is
	$-\log(2\sqrt{pq})=-\tfrac12\log(4pq)=D(\tfrac12\Vert p)$, the
	Kullback--Leibler divergence of the fair coin from $p$.
\end{corollary}

\begin{proof}
	Apply Proposition~\ref{prop:hyp-exp} with $N=v$ and $r=\tfrac12$, legitimate
	because $r=\tfrac12<p$ is fixed. Then $\theta_N=\{v/2\}=\theta_v$; the tilt is
	$\rho=r q/(\bar r p)=q/p$; the rate is
	$D(\tfrac12\Vert p)=\tfrac12\log\tfrac1{2p}+\tfrac12\log\tfrac1{2q}
	=-\tfrac12\log(4pq)$, so that $e^{-vD}=(4pq)^{v/2}=(2\sqrt{pq})^{v}$; and
	$r\bar r=\tfrac14$, so $(2\pi v\,r\bar r)^{-1/2}=\sqrt{2/(\pi v)}$.
	Substituting into \eqref{eq:hyp-asymp} gives \eqref{eq:walk-asymp}.
\end{proof}

\begin{remark}[Three readings]\label{rem:walk-reading}
	\emph{The parity is the lattice defect.} The exact value of $F$ is usually
	written out in the two parities of $v$. This split is the lattice defect of the
	fold: $\theta_v=\{v/2\}$ takes exactly the two values $0$ and $\tfrac12$, so the parity of $v$ is the
	oscillation of \S\ref{sec:hyp}, specialised to a fold at a half-integer.

	\emph{The scale changes at $p=\tfrac12$.} For $p\ne\tfrac12$ the walk travels a
	distance $\asymp v$ and the correction is exponentially small. At $p=\tfrac12$
	the rate $D(\tfrac12\Vert p)$ vanishes, the fold returns to the mean, and the
	scale collapses from $v$ to $\sqrt v$: $F(v,\tfrac12)\sim\sqrt{2v/\pi}$, which
	is De~Moivre's answer to Cuming's problem of 1721,
	$\E|S_{2n}|=2n\binom{2n}{n}2^{-2n}$. The transition is \S\ref{sec:regimes}.

	\emph{The remainder degrades near $p=\tfrac12$.} By
	Corollary~\ref{cor:parameter} the $O(v^{-1})$ in \eqref{eq:walk-asymp} is a
	fixed-$p$ remainder whose constant carries $(1-\rho)^{-2}$: numerically it is
	about $-17$ at $p=0.7$ and about $-260$ at $p=0.55$. This is the behaviour
	predicted by Corollary~\ref{cor:parameter}.
\end{remark}

% ---------------------------------------------------------------------
\subsection{The hypergeometric form, and why its orientation matters}

	The truncated binomial sum in \eqref{eq:F-cdf} is a terminating Gauss
	${}_2F_1$, equivalently an incomplete beta function \cite[\S8.17]{dlmf}; this is
	the representation Katti~\cite{katti1960} used. Writing
	$\ell:=\lfloor v/2\rfloor$, the exact value takes the form
	\begin{equation}\label{eq:F-hyp}
		F(v,p) = v(q-p) + 2 v\, p^{\ell+1} q^\ell \binom{2\ell}{\ell}
		\left[1 - \frac{q-p}{q}\cdot\frac{\ell}{\ell+1}\cdot
		{}_2F_1\!\left(1, 1{-}\ell;\, 2{+}\ell;\, -\frac{p}{q}\right)\right],
	\end{equation}
	a routine consequence of the identification of a truncated binomial sum with a
	terminating ${}_2F_1$; the second upper parameter $1-\ell$ is a nonpositive
	integer, so the series is a polynomial in $-p/q$ and no question of convergence
	arises, not even at $-p/q=-1$.

	One can try to obtain the asymptotics of $F$ from \eqref{eq:F-hyp} directly:
	one asymptotic binomial coefficient multiplies a terminating series. This route
	depends on the orientation, which may be chosen because $F(v,p)=F(v,q)$.

\begin{remark}[The two orientations]\label{rem:two-routes}
	Take $p<\tfrac12$. Then $|-p/q|<1$, the series is dominated by its first terms,
	and with
	\[
	\begin{aligned}
		H&={}_2F_1\Bigl(1,1-\ell;\ell+2;-\tfrac pq\Bigr)
		=\sum_{j\ge0}\Bigl(\frac pq\Bigr)^{j}
		\prod_{i=1}^{j}\frac{\ell-i}{\ell+1+i},\\
		\prod_{i=1}^j\frac{\ell-i}{\ell+1+i}
		&=1-\frac{j(j+2)}{\ell}+O(\ell^{-2})
	\end{aligned}
	\]
	uniformly on the relevant range, one obtains
	\[
		H=\frac{q}{q-p}
		-\frac1\ell\sum_{j\ge0}j(j+2)\Bigl(\frac pq\Bigr)^{j}+O(\ell^{-2}),
	\]
	an expansion in integer powers of $\ell^{-1}$ whose coefficients are geometric
	moments in the ratio $p/q$. Multiplying by the Stirling expansion of
	$\binom{2\ell}{\ell}p^{\ell}q^{\ell}$ then reproduces the coefficient structure
	of \S\ref{sec:hyp}. This is equivalent to the tilt construction: the argument of the hypergeometric function is the exponential
	tilt, $-p/q=-\rho$, and its geometric moments are the $M_i$ of
	Theorem~\ref{thm:all-coeff}.

	Take $p>\tfrac12$ instead. The
	hypergeometric function is then exponentially \emph{large} while the binomial
	coefficient is exponentially \emph{small}, and their rates cancel exactly:
	\[
		\binom{2\ell}{\ell}p^{\ell}q^{\ell}\cdot
		{}_2F_1\Bigl(1,1-\ell;\ell+2;-\frac pq\Bigr)\ \longrightarrow\ \frac qp
		\qquad(\ell\to\infty),
	\]
	as one sees by matching the leading term of \eqref{eq:F-hyp} against
	$F\sim v(p-q)$. (Numerically, at $p=0.6$ the product takes the values
	$0.6706$, $0.6707$, $0.6688$, $0.6677$ at $\ell=80,160,320,640$, against
	$q/p=0.6\overline{6}$; the approach is slow, of order $\ell^{-1}$.) The whole
	$O(v)$ leading term of $F$ is thus produced by a cancellation of two
	exponentials, while the exponentially small correction is not obtained from the
	resulting $\ell^{-1}$ expansion.

	The hypergeometric form is therefore useful when its argument is a contraction.
	In that orientation it is the tilt expansion; in the opposite orientation it
	places the correction behind an exponential cancellation.
\end{remark}

% =====================================================================
\section{The regimes, and the limits of the expansion}
\label{sec:regimes}

	The expansion of \S\ref{sec:hyp} is a fixed-$r$ expansion. What happens as
	$r\to p$ --- as the prescribed centre approaches the mean --- is a
	moderate-deviation question, and the answer is classical. We record it because
	it delimits the range in which \S\ref{sec:hyp} is informative. Write
	\begin{equation}\label{eq:a-def}
		a:=\frac{\sqrt N\,(p-r)}{\sqrt{pq}},\qquad \sigma:=\sqrt{Npq},
	\end{equation}
	and let $G(a):=\E|Z+a|=\sqrt{2/\pi}\,e^{-a^2/2}+a\operatorname{erf}(a/\sqrt2)$ be
	the folded-normal mean profile~\cite{leone1961}, so that
	$G(a)-a=2\,\E[(Z+a)^-]$.

\begin{proposition}[Moderate deviations]\label{prop:chain}
	Let $0<p<1$ be fixed and let $\psi$ be any function with $\psi(N)=o(N^{1/6})$.
	Then, uniformly over all $r\in(0,p]$ with $0\le a\le\psi(N)$,
	\begin{equation}\label{eq:chain}
		\E A=2N(p-r)+2\sigma\bigl(G(a)-a\bigr)\bigl(1+o(1)\bigr).
	\end{equation}
\end{proposition}

	The restriction to $r\le p$ (that is, $a\ge0$) is only for definiteness: since $G$ is
	even and the relabelling $X\mapsto N-X$, $(p,r)\mapsto(q,\bar r)$ of \S\ref{sec:hyp}
	preserves $A$ while sending $a\mapsto-a$, the statement covers $r>p$ as well.

	At $a=0$ this is $\E A=2\sigma G(0)(1+o(1))=\sqrt{2N/\pi}\,s\,(1+o(1))$, the
	leading term of \eqref{eq:mean-asymp}. For bounded $a$ it is the folded normal,
	$\E A=2\sigma G(a)(1+o(1))$. For $a\to\infty$ the Mills expansion
	$G(a)-a\sim2\varphi(a)/a^{2}$ turns \eqref{eq:chain} into
	$\E A-2N(p-r)\sim4\sigma\varphi(a)/a^{2}$, which is the correction predicted by
	\eqref{eq:hyp-asymp}. The three descriptions are obtained from the same
	moderate-deviation profile; the connection must be made through the variable
	$a$, not by substituting $r=p$ into \eqref{eq:hyp-asymp}, where $\rho\to1$ and
	the bracket diverges.

\begin{proof}
	Put $Z_N:=(X-Np)/\sigma$, so that $X-Nr=\sigma(Z_N+a)$ and, by $|x|=x+2x^-$,
	\[
		\E A=2\sigma\,\E|Z_N+a|=2N(p-r)+4\sigma\,\E\bigl[(Z_N+a)^-\bigr],
	\]
	using $2\sigma a=2N(p-r)$. Since $G(a)-a=2\,\E[(Z+a)^-]$, the assertion is that
	$\E[(Z_N+a)^-]=(1+o(1))\,\E[(Z+a)^-]$, and both sides are tail integrals,
	$\int_a^{\infty}\Pr\{Z_N\le-t\}\,dt$ and $\int_a^{\infty}\Phi(-t)\,dt$.

	Choose $T_N\to\infty$ with $\psi(N)=o(T_N)$ and $T_N=o(N^{1/6})$; this is
	possible because $\psi(N)=o(N^{1/6})$. Cram\'er's moderate-deviation theorem for
	sums of independent bounded variables
	(Petrov~\cite[Ch.~VIII, Thm.~2]{petrov1975}) is a statement about the
	distribution function and requires only Cram\'er's condition on the moment
	generating function, which a bounded summand satisfies; no non-lattice
	hypothesis enters, and the lattice character of $X$ is therefore harmless here.
	It gives
	\[
		\frac{\Pr\{Z_N\le-t\}}{\Phi(-t)}
		=\exp\Bigl\{-\frac{t^3}{\sqrt N}\,
		\lambda_{\mathrm{Cr}}\Bigl(\frac{-t}{\sqrt N}\Bigr)\Bigr\}
		\Bigl(1+O\Bigl(\frac{1+t}{\sqrt N}\Bigr)\Bigr)=1+o(1)
	\]
	uniformly for $0\le t\le T_N$, since $T_N^3/\sqrt N=o(1)$; here
	$\lambda_{\mathrm{Cr}}$ is Cram\'er's series, analytic near the origin. Hence
	\[
		\int_a^{T_N}\Pr\{Z_N\le-t\}\,dt
		=\bigl(1+o(1)\bigr)\int_a^{T_N}\Phi(-t)\,dt,
	\]
	uniformly over $0\le a\le\psi(N)$.

	For the far tail, Hoeffding's inequality gives
	$\Pr\{Z_N\le-t\}\le e^{-2pq\,t^2}$ for all $t\ge0$, so
	\[
		\int_{T_N}^{\infty}\Pr\{Z_N\le-t\}\,dt\le\frac{e^{-2pq\,T_N^2}}{4pq\,T_N},
		\qquad
		\int_{T_N}^{\infty}\Phi(-t)\,dt\le\frac{\varphi(T_N)}{T_N^{2}} .
	\]
	Since $a\le\psi(N)=o(T_N)$, both $2pq\,T_N^2-\tfrac12a^2$ and $T_N^2-a^2$ tend
	to $+\infty$; and $\int_a^{\infty}\Phi(-t)\,dt=\E[(Z+a)^-]$ is
	$\asymp\varphi(a)/a^{2}$ for large $a$ and is bounded below by a positive
	constant for bounded $a$. Both tails are therefore negligible against the main
	term in either regime, and adding the two ranges completes the proof.
\end{proof}

\begin{remark}[Why the leading coefficient alone is not enough]
	\label{rem:S0-not-enough}
	Proposition~\ref{prop:chain} settles a point that \S\ref{sec:hyp} cannot settle
	by itself. The exact correction is $\E A-2N(p-r)=4\,b(m;N,p)\,T_N$, where $T_N$
	is the \emph{whole} tail sum \eqref{eq:T-def}, whereas its leading coefficient
	is $S_0(\theta_N)$. The two are not interchangeable as $r\to p$. Comparing
	\eqref{eq:chain} with the evaluation of $4\,b(m;N,p)\,S_0(\theta_N)$ --- which,
	as $\rho\uparrow1$ with $a\to\infty$, is $4\sigma\varphi(a)a^{-2}(1+o(1))$, by
	\eqref{eq:hyp-tilt} and $1-\rho=a(Npq)^{-1/2}(1+o(1))$ --- gives
	\begin{equation}\label{eq:T-over-S0}
		\frac{T_N}{S_0(\theta_N)}\ \longrightarrow\
		\frac{G(a)-a}{2\varphi(a)/a^{2}}
		\qquad\text{whenever } a\to a^*\in(0,\infty].
	\end{equation}
	The right-hand side equals $1$ if and only if $a^*=\infty$; for finite $a^*$ it
	is strictly smaller ($0.344$, $0.629$, $0.776$ at $a^*=1,2,3$). Replacing $T_N$
	by $S_0$ is therefore justified only in the saddle regime, and the fixed-$r$
	remainder of Proposition~\ref{prop:hyp-exp} does not license it for bounded $a$.
	This is the same non-uniformity predicted by Corollary~\ref{cor:parameter}.
\end{remark}

\begin{remark}[The boundary $N^{1/6}$]\label{rem:warning}
	The exponent $\tfrac16$ in Proposition~\ref{prop:chain} is the classical
	Cram\'er boundary: it is where the cubic term $a^3/\sqrt N$ of Cram\'er's series
	ceases to be $o(1)$. It is sharp only when the leading coefficient of that
	series is nonzero, that is, when $p\ne\tfrac12$; at $p=\tfrac12$ the third
	cumulant vanishes, the cubic term is absent, and the range extends to
	$a=o(N^{1/4})$.

	Proposition~\ref{prop:chain} does not imply that $2\sigma G(a)$ remains accurate
	for a \emph{fixed} mismatch, where $a\asymp\sqrt N$. The folded normal is
	correct to leading order --- the deterministic term
	$2N(p-r)$ is right, and the discrepancy is exponentially small --- but it does
	not reproduce the exponentially small correction itself, whose ratio to the
	Gaussian prediction is unbounded. At $p=0.4$, $r=0.2$ the exact correction
	beyond $2N(p-r)$ is $2.8\cdot10^{-9}$ at $N=200$, $2.3\cdot10^{-17}$ at $N=400$
	and $2.1\cdot10^{-33}$ at $N=800$; the folded-normal expression overstates it by
	the factors $6.3$, $32.3$ and $852.7$, while \eqref{eq:hyp-asymp} is in error by
	$8.4\%$, $4.3\%$ and $2.2\%$, decaying like $N^{-1}$. The missing factor in the
	Gaussian approximation is the lattice large-deviation prefactor: in the far tail the
	tilt $\rho^{\theta_N}$ and the geometric weight $S_0(\theta_N)$ are $O(1)$
	multiplicative effects, not perturbations.

	A description that is \emph{uniform} across the whole transition, and therefore
	stronger than anything in this section, is available through Temme's uniform
	expansion of the incomplete beta function \cite[\S8.18]{dlmf},~\cite{temme2015};
	the price is that its coefficients are not elementary. The point of
	\S\ref{sec:hyp} is complementary: away from the transition, the coefficients are
	elementary and can be written down.
\end{remark}

% =====================================================================

\end{document}